\documentclass[reqno]{amsart}
\usepackage{amssymb,latexsym}
\usepackage{amsmath}
\usepackage{graphicx}
\usepackage{amscd}
\usepackage{color}
\usepackage{enumerate}
\newenvironment{enumeratei}{\begin{enumerate}[\upshape (i)]}{\end{enumerate}}
\numberwithin{equation}{section}
\theoremstyle{plain}
 \newtheorem{theorem}{Theorem}[section]
 \newtheorem{lemma}[theorem]{Lemma}
 \newtheorem{proposition}[theorem]{Proposition}
 
 \newtheorem{corollary}[theorem]{Corollary}
 
\theoremstyle{definition}
 \newtheorem{definition}[theorem]{Definition}
 \newtheorem{remark}[theorem]{Remark}
\theoremstyle{remark}
 \newtheorem{case}{Case}

\newcommand \datum {\hfill January 4, 2018}

\newcommand \Con  {\textup{Con}}
\newcommand \Sub  {\textup{Sub}}
\newcommand \ncsl  {\textup{N}\kern-0.8pt\textup{C}\textup{S}\kern-0.2pt\kern-0.3pt\textup{L}}
\newcommand \ncv[1]  {\textup{N}\kern-0.8pt\textup{C}(#1,}
\newcommand \var[1] {\mathcal{#1}}
\newcommand \slvar {\textup{SLat}_{\kern-1pt\wedge}}

\newcommand \plu [1] {#1^{+}}
\newcommand \jplu {\tuple{\plu S;\vee}}
\newcommand \sla {\tuple{S;\wedge}}
\newcommand \subplu {\Sub(\plu S;\vee)}
\newcommand \pset [1] {P(#1)}
\newcommand \tcon {\boldsymbol\tau}
\newcommand \ubta {ubt-antichain}
\newcommand \blokk [2] {#1/#2}
\newcommand \tbl [1] {\blokk {#1}\Theta}

\newcommand\ideal[1]{\mathord\downarrow #1}
\newcommand\filter[1]{\mathord\uparrow #1}

\newcommand \tbf [1] {\textbf{#1}} 
\newcommand \set[1] {\{#1\}}
\newcommand \tuple [1] {\langle #1\rangle}
\newcommand \pair [2] {\tuple{#1,#2}}

\newcommand \red [1] {\color{red}#1\color{black}}

\newcommand\url [1] {{\texttt{#1}}}

\begin{document}
\title
[Finite semilattices with many congruences]
{Finite semilattices with many congruences}

\author[G.\ Cz\'edli]{{G\'abor Cz\'edli}}
\email{czedli@math.u-szeged.hu}
\urladdr{http://www.math.u-szeged.hu/~czedli/}
\address{Bolyai Institute, University of Szeged, Hungary 6720}

\begin{abstract} For an integer $n\geq 2$, let $\ncsl(n)$ denote the set of sizes of congruence lattices of $n$-element semilattices. We find the four largest numbers belonging to $\ncsl(n)$, provided that $n$ is large enough to ensure that $|\ncsl(n)|\geq 4$. Furthermore, we describe the $n$-element semilattices witnessing these numbers.
\end{abstract}


\thanks{This research was
supported by
the Hungarian Research Grant KH 126581}

\maketitle
\section{Introduction and motivation}
The present paper is primarily motivated by a problem on tolerance relations of lattices  raised by Joanna Grygiel in her conference talk in September, 2017, which was a continuation of G\'ornicka, Grygiel, and Tyrala~\cite{GGrygielT}. Further motivation is supplied by Cz\'edli~\cite{czedliMC}, Cz\'edli and  Mure\c san~\cite{czgmuresan}, Kulin and Mure\c san~\cite{kulinmuresan}, and Mure\c san~\cite{muresan2017arXiv}, still dealing with lattices rather than semilattices. 

As usual, $\Con(A)$ will stand for the \emph{lattice of congruences} of an algebra $A$. Given a natural number $n\geq 2$ and a variety $\var V$ of algebras, the task of
\begin{equation}
\parbox{8 cm}{finding the \emph{small} numbers in the set
$\ncv{\var V}n):=\set{|\Con(A)|: A\in \var V\text{ and }|A|=n }$
and  \emph{describing} the algebras $\var V$ witnessing these numbers}
\label{eqpbxsTrG}
\end{equation}
has already deserved some attention for various varieties $\var V$, because the description of the simple $n$-element algebras in $\var V$ for various varieties $\var V$ and, in particular, even the Classification of Finite Simple Groups belong to \eqref{eqpbxsTrG} in some vague sense.  The present paper addresses an analogous problem, which is obtained from \eqref{eqpbxsTrG} by changing ``small'' to ``large''. 
Of course, this problem is hopeless for an arbitrary variety $\var V$. However, if $\var V$ is the variety $\slvar$ of \emph{meet semilattices}, then we can benefit from  Freese and Nation's classical description of the congruence lattices of finite members of $\slvar$; see \cite{freesenation}. Let us fix the following notation
\begin{equation}
\ncsl(n):=\ncv{\slvar}n)=|\set{\Con(S):S\in\slvar\text{ and }|S|=n}|;
\end{equation}
the acronym $\ncsl$ comes from ``Number of Congruences of SemiLattices''.
Our target is to determine the four largest numbers belonging to $\ncsl(n)$ and, in addition, to describe the $n$-element semilattices witnessing the these numbers.

\subsection*{Outline}
The rest of the paper is structured as follows. In Section~\ref{sectmainresult}, we introduce a semilattice construction, and we use this construction in formulating the main result, Theorem~\ref{thmmain}, to realize our target mentioned above. This section concludes with a corollary stating that a semilattice with sufficiently many  congruences is planar.  Section~\ref{sectproof} is devoted to the proof of this theorem.

\begin{figure}[ht] 
\centerline
{\includegraphics[scale=1.0]{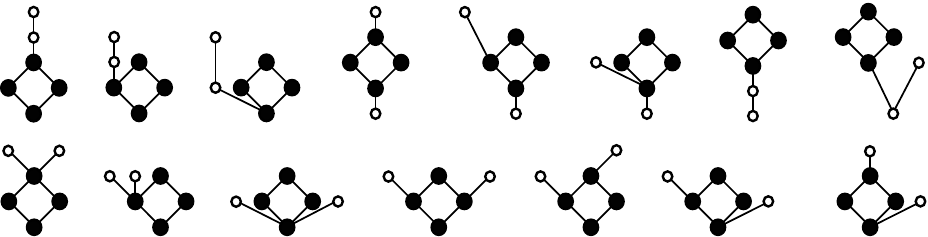}}
\caption{The full list of 6-element meet semilattices with exactly $28=28\cdot 2^{6-6}$ many congruences}
\label{figone}
\end{figure}

\section{Quasi-tree semilattices and our theorem}~\label{sectmainresult} We  follow the standard terminology and notation; see, for example, Gr\"atzer~\cite{ggfoundbook} and \cite{ggCFL2}. 
Even without explicitly saying so all the time, by a \emph{semilattice} we always mean a \emph{finite meet} semilattice $S$, that is, a finite member of $\slvar$. Such an $S=\sla$ has a least element $0=\bigwedge S$. We always denote $S\setminus \set 0$ by $\plu S$. 
If no two incomparable elements of $S$ has an upper bound, then $S$ is called a \emph{tree semilattice}. 

Next, for a meet semilattice $S$,  the congruence $\tcon=\tcon(S;\wedge)$ generated by 
\begin{equation}\set{\pair {a\wedge b}{a\vee b}: a,b\in\plu S,\text{ } a\parallel b,\text{ and } a\vee b\text{ exists in }\jplu}
\label{eqtrCnGr}
\end{equation}
will be called the \emph{tree congruence} of $\sla$. 
Of course, we can write $a,b\in S$ instead of $a,b\in \plu S$ above. Observe that for $a,b\in \plu S$,
\begin{equation}
\text{$\set{a,b}$ has an upper bound in $S$ iff $a\vee b$ exists in $\jplu$;}
\label{eqpbxCpBndxShR}
\end{equation}
hence instead of requiring the join $a\vee b\in\jplu$, it suffices to require an upper bound of $a$ and $b$ in \eqref{eqtrCnGr}. The name ``tree congruence'' is explained by the following easy statement, which will be proved in Section~\ref{sectproof}.

\begin{figure}[ht] 
\centerline
{\includegraphics[scale=1.0]{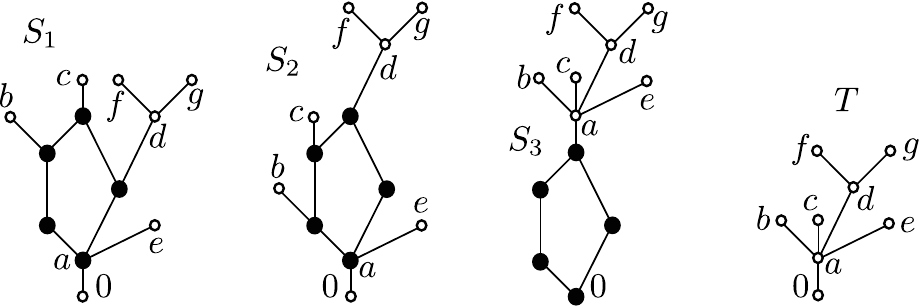}}
\caption{Three twelve-element meet semilattices with the same skeleton $T$ and the same number, 
 $26\cdot 2^{12-6}=1664$, of congruences}
\label{figtwo}
\end{figure}

\begin{proposition}\label{proptrCngR} For an arbitrary meet semilattice $\sla$, 
the quotient meet semilattice $\sla/\tcon$ is a tree. 
\end{proposition}

\begin{definition}\label{defQTRsL} By a \emph{quasi-tree semilattice} we mean a finite meet semilattice $\sla$ such that 
its tree congruence $\tcon=\tcon(S;\wedge)$ has exactly one nonsingleton block.  If $\sla$ is a quasi-tree semilattice, then the unique nonsingleton block of $\tcon$, which is a meet semilattice,  and the quotient semilattice $\sla/\tcon$ are called the \emph{nucleus} and the \emph{skeleton} of  $\sla$.
\end{definition}

Some quasi-tree semilattices are shown in Figures~\ref{figone},  \ref{figtwo}, and \ref{figthree}.
In these figures, the elements of the nuclei are the black-filled ones, while the empty-filled smaller circles stand for the rest of elements.
Although a quasi-tree semilattice $\sla$ is not determined by its skeleton and nucleus in general, the skeleton and the nucleus together carry a lot of information on $\sla$. In order to make the numbers occurring in the following theorem easy to compare, we give them in a redundant way as multiples of $2^{n-6}$.

\begin{theorem}\label{thmmain} If $\sla$ is a finite meet semilattice of size $n=|S|>1$, then the following hold.
\begin{enumeratei}
\item\label{thmmaina} $\sla$ has at most $2^{n-1}=32\cdot 2^{n-6}$ many congruences. Furthermore, we have that $|\Con(S;\wedge)|=2^{n-1}$  if and only if $\sla$ is a tree semilattice. 
\item\label{thmmainb} If $\sla$ has \emph{less} than  $2^{n-1}=32\cdot 2^{n-6}$  congruences, then it has at most $28\cdot 2^{n-6}$  congruences. Furthermore, $|\Con(S;\wedge)|=28\cdot 2^{n-6}$ if and only if $\sla$ is a quasi-tree semilattice 
and its nucleus is the four-element boolean lattice; see Figure~\textup{\ref{figone}} for $n=6$.
\item\label{thmmainc} If $\sla$ has \emph{less} than  $28\cdot 2^{n-6}$  congruences, then it has at most $26\cdot 2^{n-6}$  congruences. Furthermore, $|\Con(S;\wedge)|=26\cdot 2^{n-6}$ if and only if $\sla$ is a quasi-tree semilattice such that its nucleus is the pentagon $N_5$; see Figure~\textup{\ref{figfour}} and $S_1,\dots, S_3$ in Figure~\textup{\ref{figtwo}}.
\item\label{thmmaind} If $\sla$ has \emph{less} than  $26\cdot 2^{n-6}$  congruences, then it has at most $25\cdot 2^{n-6}$  congruences.  Furthermore, $|\Con(S;\wedge)|=25\cdot 2^{n-6}$ if and only if $\sla$ is a quasi-tree semilattice such that its nucleus is either  $F$, or  $N_6$; see Figure~\textup{\ref{figfour}} and $S_4,\dots, S_7$ in  Figure~\textup{\ref{figthree}}.
\end{enumeratei}
\end{theorem}

\begin{figure}[ht] 
\centerline
{\includegraphics[scale=1.0]{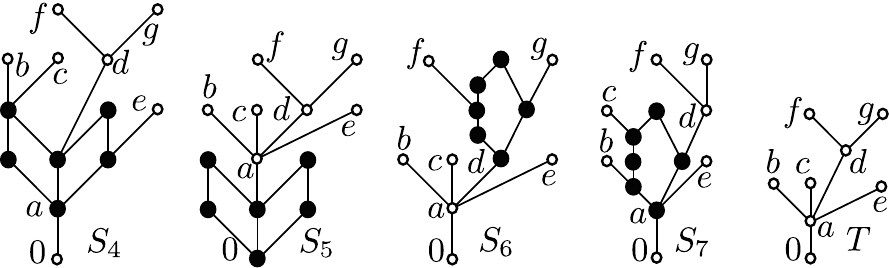}}
\caption{Four thirteen-element meet semilattices with the same skeleton $T$ and the same number, $25\cdot 2^{13-6}=3200$, of congruences}
\label{figthree}
\end{figure}

\begin{remark}\label{remarksmLsSm} Although Theorem~\ref{thmmain} holds for all $n\geq 2$, neither it gives the \emph{four largest} numbers of $\ncsl(n)$, nor it says too much for $n\leq 5$. For example,  $25\cdot 2^{n-6}$ is not even an integer if $n\leq 5$. Hence, we note the following facts without including their trivial proofs in the paper.
\begin{enumerate}[\upshape\quad (A)]
\item $\ncsl(2)=\set{2=2^{2-1}}$
\item $\ncsl(3)=\set{4=2^{3-1}}$
\item $\ncsl(4)=\set{8=2^{4-1}, \,\, 7=28\cdot 2^{4-6}}$
\item $\ncsl(5)=\set{16=2^{5-1}, \,\, 14=28\cdot 2^{5-6}, \,\, 13=26\cdot 2^{5-6}, \,\,12 }$.
Note that $12$ is witnessed by $M_3=\tuple{M_3,\wedge}$; see Figure~\ref{figfour}.
\end{enumerate}
\end{remark}

\begin{figure}[ht] 
\centerline
{\includegraphics[scale=1.0]{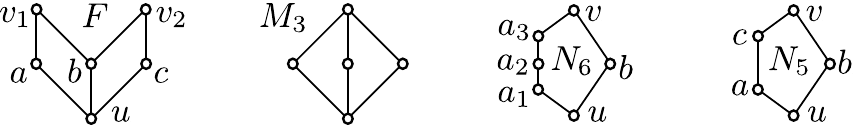}}
\caption{$F$, $M_3$,  $N_6$, and the pentagon, $N_5$}
\label{figfour}
\end{figure}

A semilattice is \emph{planar} if it has a planar Hasse diagram, that is a Hasse diagram in which edges can intersect only at their endpoints, that is, at vertices.  Theorem~\ref{thmmain} immediately implies the following statement.

\begin{corollary}\label{corolplanar} If an $n$-element meet semilattice has at least $25\cdot 2^{n-6}$ congruences, then it  is planar. 
\end{corollary}

The following statement is due to Freese~\cite{freesecomp}; see also Cz\'edli~\cite{czedliMC} for a second proof, which gives the first half of the following corollary  for arbitrary finite algebras in congruence distributive varieties, not only for lattices.

\begin{corollary}\label{coroloWBt} For every $n$-element lattice $L$, we have that $|\Con(L)|\leq 2^{n-1}$. Furthermore, $|\Con(L)|= 2^{n-1}$ if and only if $L$ is a chain.
\end{corollary}

As a preparation for a remark below, we derive this corollary from Theorem~\ref{thmmain}\eqref{thmmaina}
here rather than in the next section.

\begin{proof}[Proof of Corollary~\textup{\ref{coroloWBt}}]
The only $n$-element tree semilattice that is also a lattice is the $n$-element chain. For an equivalence relation $\Theta$ on this chain $\tuple{C;\leq}$, 
\begin{equation}
\parbox{6.2cm}{$\Theta\in\Con(C;\wedge)$ iff $\Theta\in\Con(C;\vee,\wedge)$ iff every $\Theta$-block is an interval of $\tuple{C;\leq}$.}
\label{eqpbxdzRtsqzp}
\end{equation}
Observe that every $\Theta\in\Con(L;\vee,\wedge)$ also belongs to $\Con(L;\wedge)$. Hence, using Theorem~\ref{thmmain}\eqref{thmmaina} at $\leq^\ast$ below, we obtain that 
\[
|\Con(L;\vee,\wedge)|\leq  |\Con(L;\wedge)| \leq^\ast |\Con(C;\wedge)| = |\Con(C;\vee,\wedge)|,
\]
proving Corollary~\ref{coroloWBt}.
\end{proof}

Next, we point out that Theorem~\ref{thmmain}\eqref{thmmaina} 
plays an essential role in the proof above. 

\begin{remark} 
The second part of  \eqref{eqpbxdzRtsqzp} might give the false feeling that Szpilrajn's Extension Theorem \cite{szpilrajn} in itself implies  Corollary~\ref{coroloWBt} as follows: extend the ordering relation of $L$ to a linear ordering to obtain a chain; then we obtain more intervals and thus more equivalences whose blocks are intervals, and so more congruences by \eqref{eqpbxdzRtsqzp}. 
In order to point out that this argument does not work, 
let $\tuple{L;\leq_1}$ be the direct product of the two-element chain and the three-element chain. Although $\leq_1$ can be extended to a linear ordering $\leq_2$ and the chain $\tuple{L;\leq_2}$ has more intervals than  $\tuple{L;\leq_1}$, the lattice 
$\tuple{L;\leq_1}$ has 34 equivalences whose blocks are intervals but the chain  $\tuple{L;\leq_2}$ has only 32.
\end{remark}

\section{Proofs}\label{sectproof}

\begin{proof}[Proof of Proposition~\textup{\ref{proptrCngR}}]
A subset $X$ of $\sla$ is said to be \emph{convex}, if $x<y<z$ and $x,z\in X$ imply that $y\in X$, for any $x,y,z\in S$.
It is well known that 
\begin{equation}
\text{the blocks of every congruence of $\sla$ are convex subsets.}
\label{eqtxtconblconVex}
\end{equation}
Indeed, if $\Theta\in\Con(S;\wedge)$, $x\leq y\leq z$ and $\pair x z\in\Theta$, then $\pair {x}{y}=\pair {x\wedge y}{z\wedge y}\in\Theta$, whereby $y\in\blokk x\Theta$, which shows \eqref{eqtxtconblconVex}. By \eqref{eqtxtconblconVex},  the $\tcon$-blocks are convex subsets of $\sla$. 
Next, for the sake of contradiction, suppose that $a,b\in S$ such that $\blokk a\tcon$ and $\blokk b\tcon$ are incomparable elements of the meet semilattice $\sla/\tcon$ such that $\blokk c\tcon\in \sla/\tcon$ is an upper bound of them. Let $a':=a\wedge c$ and $b':=b\wedge c$ in $\sla$. 
Since $\blokk a\tcon\leq \blokk c\tcon$, we have that $\blokk a\tcon =\blokk a\tcon\wedge \blokk c\tcon=\blokk {(a\wedge c)}\tcon=\blokk {a'}\tcon$, whence  $\pair a{a'}\in\tcon$. Similarly, $\pair b{b'}\in\tcon$. Since $a'\leq c$ and $b'\leq c$, \eqref{eqpbxCpBndxShR} implies the existence of $a'\vee b'\in \jplu$. Hence, by the definition of $\tcon$, we have that $\pair{a'\wedge b'}
{a'\vee b'}\in\tcon$. Since the $\tcon$-block $\blokk{(a'\wedge b')}\tcon$ is convex, $\pair{a'}{b'}\in\tcon$. Combining this with $\pair a{a'}\in\tcon$ and $\pair b{b'}\in\tcon$, we obtain that $\pair a b\in\tcon$. Hence, $\blokk a\tcon$ equals $\blokk b\tcon$, which contradicts their incomparability.
\end{proof}

Note that, in general, $\tcon=\tcon(S;\wedge)$ is not the smallest congruence of $\sla$ such that $\sla/\tcon$ is a tree; this is exemplified by the semilattice reduct of the four-element boolean lattice.

The proof of Theorem~\ref{thmmain} will be divided into several lemmas, some of them being interesting in themselves, and we are going to prove parts \eqref{thmmaina}--\eqref{thmmaind} separately.

Remember that, for a finite meet semilattice $S=\sla$, we use the notation $\plu S:=S\setminus\set{0}$. Then $\jplu$ is a partial algebra, which we call the \emph{partial join-semilattice} associated with $S$. By a \emph{partial subalgebra} of $\jplu$ we mean a subset $X$ of $\plu S$ such that whenever $x,y\in \plu S$ and $x\vee y$ is defined, then $x\vee y\in \plu S$. The set of all partial subalgebras of $\jplu$ form a lattice, which we denote by $\subplu$. For convenience, our convention is that $\emptyset\in \subplu$. The proof of Theorem~\ref{thmmain} relies on the following result of Freese and Nation~\cite{freesenation}.

\begin{lemma}[{Freese and Nation~\cite[Lemma 1]{freesenation}}]\label{lemmafrnat} For every finite meet semilattice $\sla$, the lattice $\Con(S;\wedge)$ is dually isomorphic to $\subplu$. In particular, we have that $|\Con(S;\wedge)|=|\subplu|$.
\end{lemma}

Note that Freese and Nation~\cite{freesenation} uses $\Sub(S;\vee,0)$, which does not contain the emptyset, but the  natural isomorphism from $\subplu$ onto $\Sub(S;\vee,0)$, defined by $X\mapsto X\cup\set 0$,  allows us to cite their result in the above form.
The following lemma is almost trivial; having no reference at hand, we are going to present a short proof. As usual, \emph{intervals} are nonempty subsets of the form $[a,b]:=\set{x: a\leq x\leq b}$. The \emph{principal ideal} and the \emph{principal filter} generated by an element $a\in S$ are denoted by $\ideal a=\set{x\in S: x\leq a}$ and $\filter a=\set{x\in S: a\leq x}$, respectively.  Meet-closed convex subsets are \emph{convex subsemilattices}. 
A subsemilattice is \emph{nontrivial} if it consists of at least two elements.

\begin{lemma}\label{lemmaCnGrtS} Let $X$ be a nontrivial convex subsemilattice of  a finite semilattice $\sla$, and denote the smallest element of $X$ by $u:=\bigwedge X$. Then the following two conditions are equivalent. 
\begin{enumerate}[\upshape (a)]
\item\label{lemmaCnGrtSa} The equivalence $\Theta$ on $S$ whose only nonsingleton block is $X$ is a congruence of $\sla$.
\item\label{lemmaCnGrtSb} For all $c\in S\setminus \filter u$ and  every maximal element $v$ of $X$, we have that  $u\wedge c=v\wedge c$.
\end{enumerate}
\end{lemma}

\begin{proof}[Proof of Lemma~\textup{\ref{lemmaCnGrtS}}]
Assume \eqref{lemmaCnGrtSa} and let $c\notin \filter u$, and let $v$ be a maximal element of $X$. Then $c\notin \filter v$,   $u\nleq u\wedge c$, and  $u\nleq v\wedge c$. Hence, none of  $u\wedge c$ and  $v\wedge c$ is in $X$, but these two elements are collapsed by $\Theta$ since $\pair u v\in\Theta$. Thus, the definition of $\Theta$ gives that  $u\wedge c=v\wedge c$, proving that \eqref{lemmaCnGrtSa} implies \eqref{lemmaCnGrtSb}.

Next, assume \eqref{lemmaCnGrtSb}, and let $\Theta$ be defined as in \eqref{lemmaCnGrtSa}. First, we show that for all $x,y,z\in S$,
\begin{equation}
\text{if $\pair x y\in\Theta$, then $\pair{x\wedge z}{y\wedge z}\in \Theta$.}
\label{eqszmpRthGmX}
\end{equation}
This is trivial for $x=y$, so we can assume that $x,y\in X$. Pick maximal elements $x_1$ and $y_1$ in $X$ such that $x\leq x_1$ and $y\leq y_1$. First, let $z\in \filter u$.  Then, using the convexity of $X$,
$x\wedge z\in [u,x]\subseteq X$ and, similarly, $y\wedge z\in X$, whence we obtain that $\pair{x\wedge z}{y\wedge z}\in \Theta$ by the definition of $\Theta$.  Second, let $z\in S\setminus\filter u$. Then $x\wedge z$ belongs to the interval $[u\wedge z, x_1\wedge z]$, which is the singleton set $\set{u\wedge z}$ by  \eqref{lemmaCnGrtSb}. Hence, $x\wedge z=u\wedge z$. Similarly, $y\wedge z=u\wedge z$, whereby $\pair{x\wedge z}{y\wedge z}\in \Theta$.   Thus, \eqref{eqszmpRthGmX} holds. 

Finally, if $\pair{x_1}{y_1}\in\Theta$ and $\pair{x_2}{y_2}\in\Theta$, then we obtain from \eqref{eqszmpRthGmX} that both $\pair{x_1\wedge x_2}{y_1\wedge x_2}$ and
$\pair{y_1\wedge x_2}{y_1\wedge y_2}$ belong to $\Theta$, whereby
transitivity gives that $\pair{x_1\wedge x_2}{y_1\wedge y_2}\in \Theta$. 
Consequently, $\Theta$ is a congruence and \eqref{lemmaCnGrtSb} implies \eqref{lemmaCnGrtSa}.
\end{proof}

The \emph{powerset} of a set $A$ will be denoted by $\pset A=\set{X: X\subseteq A}$. In the rest of the paper, 
\begin{equation}
\parbox{8cm}{
$n\geq2$ denotes a natural number,   $\sla$ will stand for an $n$-element meet semilattice, and we will also use the notation
\hfill$k:=|\Con(S;\wedge)|=|\subplu|$;}
\label{eqwhatisk}
\end{equation}
here the second equality is valid by Lemma~\ref{lemmafrnat}.

\begin{proof}[Proof of Theorem~\textup{\ref{thmmain}\eqref{thmmaina}}] 
Since $|\plu S|=n-1$,  $\plu S$ has at most $2^{n-1}$ subsets, whereby 
$|\Con(S;\wedge)|=k\leq |P(\plu S)|=2^{n-1}$, as required. If $\sla$ is a tree semilattice, then $x\vee y$ is defined only if $x$ and $y$ form a comparable pair of $\plu S$, whence $x\vee y\in\set{x,y}$. Hence, every subset of $\plu S$ belongs to $\subplu$, and so $k=|\subplu|=|\pset{\plu S}|=2^{n-1}$. 
If $S$ is not a tree semilattice, then there is a pair $\pair a b$ of incomparable elements of $\plu S$ with an upper bound.  By \eqref{eqpbxCpBndxShR}, 
$a\vee b$ is defined in $\jplu$. Hence, $\set{a,b}\notin \subplu$ and so 
$k=|\subplu|<|\pset{\plu S}|=2^{n-1}$. This completes the proof of part \eqref{thmmaina}.
\end{proof}

By an \emph{upper bounded two-element antichain}, abbreviated as   \emph{\ubta}, we mean a two-element subset $\set{x, y}$ of a finite meet semilattice $\sla$ such that $x\parallel y$ and $\filter x\cap\filter y\neq\emptyset$. By \eqref{eqpbxCpBndxShR}, every \ubta{} $\set{x, y}$ has a join in $\plu S$ but this join is outside $\set{x, y}$. Therefore, 
\begin{equation}
\text{$\subplu$ contains no \ubta.}
\label{eqtxtnoubta}
\end{equation}
Besides \eqref{eqtxtnoubta}, the importance of \ubta{}s is explained by the following lemma.

\begin{lemma}\label{lemmauGbtThnNZcls}
 Let $X$ be a convex subsemilattice of  a finite semilattice $\sla$  such that $|X|\geq 2$ and $X\times X\subseteq\tcon$; see \eqref{proptrCngR}. If $X$ contains all \ubta{}s $\set{p,q}$ of $\sla$ together with their joins $p\vee q$, then $\sla$ is a quasi-tree semilattice and its nucleus is $X$.
\end{lemma}

\begin{proof}[Proof of Lemma~\textup{\ref{lemmauGbtThnNZcls}}]
Denote the smallest element of $X$ by $u:=\bigwedge X$.
Let $\Theta$ be the equivalence relation on $S$ with $X$ as the only nonsingleton block of $\Theta$. In order to prove that $\Theta\in\Con(S;\wedge)$, assume that $c\in S\setminus \filter u$ and $v$ is a maximal element of $X$. For the sake of contradiction, suppose that $u\wedge c\neq v\wedge c$, which means that $u\wedge c < v\wedge c$. If we had that $v\wedge c\leq u$, then  $v\wedge c =u\wedge (v\wedge c)=(u\wedge v)\wedge c=u\wedge c$ would be a contradiction. Thus, $v\wedge c\nleq u$.  
On the other hand, $u\nleq v\wedge c$ since $u\nleq c$, whereby 
 $u \parallel v\wedge c$. Since $v$ is a common upper bound of $u$ and $v\wedge c$, we obtain that $\set{u,v\wedge c}$ is a \ubta{}. This is a contradiction since $c\notin\filter u$ implies that $u\nleq v\wedge c$, whence 
the \ubta{}  $\set{u,v\wedge c}$ is not a subset of $X$. Hence, $u\wedge c = v\wedge c$, and it follows from Lemma~\ref{lemmaCnGrtS} that $\Theta\in\Con(S;\wedge)$. 

Next, in order to show that $\sla/\Theta$ is a tree, suppose the contrary. 
Then there are two incomparable $\Theta$-blocks $\tbl x$ and $\tbl y$ that have an upper bound $\tbl z$. Since $u\in X$  and all other $\Theta$-blocks are singletons, every $\Theta$-block has a smallest element. This fact allows us to assume that each of $x$, $y$, and $z$ is the least element of its $\Theta$-block. Since $\tbl x\leq \tbl z$, we have that 
$\tbl x=\tbl x\wedge \tbl z=\tbl{(x\wedge z)}$, that is, $\pair x{x\wedge z}\in\Theta$. But the least element of $\tbl x$ is $x$, whence $x=x\wedge z$, that is, $x\leq z$. We obtain similarly that $y\leq z$, that is, $\set{x,y}$ has an upper bound, $z$. 
Since $x\wedge y=x$ would imply that $\tbl x\wedge \tbl y=\tbl{(x\wedge y)}=\tbl x$, contradicting that $\set{\tbl x,\tbl y}$ is an antichain, we obtain that $x\nleq y$. We obtain $y\nleq x$ similarly. Thus, $\set{x,y}$ is a \ubta{}, whereby
 $\set{x,y}\subseteq X$. But then $\tbl x=X=\tbl y$, contradicting the initial assumption that these two $\Theta$-blocks are incomparable. Therefore, $X/\Theta$ is a tree. Hence, in order to complete the proof, we need to show that $\Theta=\tcon$. Since $X\times X\subseteq \tcon$, the inclusion $\Theta\subseteq \tcon$ is clear. In order to see the converse inclusion, let $\pair {a\wedge b}{a\vee b}$ be a pair occurring in \eqref{eqtrCnGr}. Then $\set{a,b}$ is a \ubta{}, so $\set{a,b}\subseteq X$ and, by the assumptions of the lemma, both $a\vee b$ and $a\wedge b$ belong to $X$. Hence, the pairs in \eqref{eqtrCnGr}  are collapsed by $\Theta$ and we conclude that $\tcon\subseteq \Theta$. Hence, $\Theta=\tcon$, and 
the proof of Lemma~\ref{lemmauGbtThnNZcls} is complete.
\end{proof}

\begin{lemma}\label{lemmabtggQPx} If 
$\sla$ from \eqref{eqwhatisk} contains exactly one \ubta{},
then  $\sla$ is a quasi-tree semilattice and its nucleus is the four-element boolean lattice.
\end{lemma}

\begin{proof}[Proof of Lemma~\textup{\ref{lemmabtggQPx}}]
Let us denote by $\set{a,b}$ the unique \ubta{} of $\sla$.
Let $v:=a\vee b$, which exists by \eqref{eqpbxCpBndxShR}, and let $u:=a\wedge b$. Then $L:=[u,v]$ contains every \ubta. 
Since $\pair u v\in\tcon$ by \eqref{eqtrCnGr} and the $\tcon$-blocks are convex, $L\times L\subseteq \tcon$.
With reference to Lemma~\ref{lemmauGbtThnNZcls}, it suffices to show that $L$ is the four-element boolean lattice. In fact, it suffices to show that $L\subseteq\set{u,a,b,v}$ since the converse inclusion is evident. Suppose the contrary, and let $x\in L\setminus \set{u,a,b,v}$. If $x\parallel a$, then $\set{a,x}$ is a 
\ubta{} (with upper bound $v$) but it is distinct from $\set{a,b}$, which contradicts the fact that $\set{a,b}$ is the only \ubta. Hence, $a$ and $x$ and comparable. We obtain similarly that $b$ and $x$ are comparable. If $x\leq a$ and $x\leq b$, then $u\leq x\leq a\wedge b=u$ leads to $x=u\in L$, which is not the case. We obtain dually that 
the conjunction of  $x\geq a$ and $x\geq b$ is impossible. Hence, 
$a\leq x \leq b$ or $b\leq x\leq a$, contradicting that $\set{a,b}$ is an antichain. This shows that $L\subseteq\set{u,a,b,v}$, completing the proof of Lemma~\ref{lemmabtggQPx}.
\end{proof}

\begin{proof}[Proof of Theorem~\textup{\ref{thmmain}\eqref{thmmainb}}]
Assume that $k<2^{n-1}$; see \eqref{eqwhatisk}. By Theorem~\ref{thmmain}\eqref{thmmaina},  $S$ is not a tree. Hence, $n=|S|\geq 4$.  Since $|\subplu|=k<2^{n-1}=|\pset{\plu S}|$, not every subset of $\plu S$ is $\vee$-closed. Thus,  we can pick $a,b\in \plu S$ such that $a\parallel b$ and $a\vee b$ exists in $\jplu$. Since $|\plu S\setminus\set{a,b,a\vee b}|=2^{n-4}$, there are $2^{n-4}$ subsets of $\plu S$ that contain $a$, $b$, but not $a\vee b$; these subsets do not belong to $\subplu$. Thus, $k\leq 2^{n-1}- 2^{n-4}=32\cdot 2^{n-6}-4\cdot 2^{n-6}=28\cdot 2^{n-6}$, proving the first half of \eqref{thmmainb}. 

Next, assume that  $k=28\cdot 2^{n-6}$ and choose $a$ and $b$ as above.  There are $2^{n-4}=4\cdot 2^{n-6}$ subsets of $\plu S$ containing $a$ and $b$, but not containing $a\vee b$; these subsets are not in $\jplu$. Thus, all the remaining  $32\cdot 2^{n-6} - 4\cdot 2^{n-6} = 28\cdot 2^{n-6}$ subsets belong to $\jplu$ since $k=28\cdot 2^{n-6}$. 
In particular, for every \ubta{} $\set{x,y}$, we have that $\set{x,y}\neq\set{a,b}\Rightarrow \set{x,y}\in\subplu$. This implication and \eqref{eqtxtnoubta} yield that $\set{a,b}$ is the only \ubta{} in $\sla$.
Thus, it follows from Lemma~\ref{lemmabtggQPx} that $\sla$ is a quasi-tree semilattice of the required form. 

Conversely, assume that $\sla$ is of the form described in \ref{thmmain}\eqref{thmmainb}. Choosing  the notation so that its nucleus is $\set{a\wedge b, a,b,a\vee b}$, the only \ubta{} is $\set{a,b}$, whence a subset $X$ of $\plu S$ is not in $\subplu$ iff $a,b\in X$ but $a\vee b\notin X$. 
There are $2^{n-4}=4\cdot 2^{n-6}$ such subsets $X$, and we obtain that 
$k=|\subplu|=|\pset{\plu S}| -4\cdot 2^{n-6}=32\cdot2^{n-6}-4\cdot 2^{n-6}=28\cdot 2^{n-6}$, as required.
This completes the proof of  Theorem~\textup{\ref{thmmain}\eqref{thmmainb}}.
\end{proof}

\begin{lemma}\label{lemmabtwWsPtZ} If 
$\sla$ from \eqref{eqwhatisk} contains exactly two \ubta{}s,
$\set{ a, b}$ and $\set{ c, b}$ such that $a<c$, then  $\sla$ is a quasi-tree semilattice and its nucleus is the pentagon lattice $N_5$.
\end{lemma}

\begin{proof}[Proof of Lemma~\textup{\ref{lemmabtwWsPtZ}}] By \eqref{eqpbxCpBndxShR}, we can let $v:=a\vee b$. 
Since $v\leq c$ would lead to $b\leq c$, we have that $v\nleq c$. In particular, $v\neq c$, and we also have that $v\notin\set{a, b}$ since $\set{a, b}$ is an antichain. Thus, $\set{c,v}
$ is a two-element subset of $S$ and it is distinct from $\set{ a, b}$ and $\set{ a, c}$. Hence,  $\set{c,v}$ is not a \ubta. Since $b\vee c$, which exists by \eqref{eqpbxCpBndxShR}, is clearly an upper bound of 
$\set{c,v}$, it follows that $\set{c,v }$ is not an antichain. This fact and $v\nleq c$ yield that $c\leq v$. Thus, $v=a\vee b\leq c\vee b\leq v$, that is, $v=a\vee b=a\vee c$.  Next, let $u:=b\wedge c$; clearly, $u\notin \set{b,c}$. If we had that $a\parallel u$, then $\set{a,u}$ would be a third \ubta{} (with upper bound $c$), whence $a$ and $u$ are comparable elements.  
Since $a\leq u$ would lead to $a\leq b$ by transitivity, we have that $u\leq a$. Hence, $u\leq a\wedge b\leq c\wedge b=u$, and so $a\wedge b=u$. The equalities established so far show that $L:=\set{u,a,b,c,v}$ is a sublattice isomorphic to $N_5$. In order to show that $L$ is the interval $[u,v]$, suppose the contrary, and let $x\in [u,v]\setminus L$. If $x\parallel b$, then $\set{b,x}$ would be a third \ubta{} (with upper bound $v$), which would be a contradiction. If we had that $b<x<v$, then $\set{c,x}$ would be a \ubta, a contradiction. Similarly, $a<x<b$ gives that   $\set{a,x}$  is a \ubta, a contradiction again. Thus, $L=[u,v]$ is an interval of $S$. By \eqref{eqtrCnGr},
$\pair u v=\pair {a\wedge b}{a\vee b}\in\tcon$. Using that the $\tcon$-blocks are convex subsets, we obtain that 
$L\times L=[u,v]\times [u,v]\subseteq \tcon$. Thus, 
Lemma~\ref{lemmabtwWsPtZ} follows from 
Lemma~\ref{lemmauGbtThnNZcls}.
\end{proof}

\begin{proof}[Proof of Theorem~\textup{\ref{thmmain}\eqref{thmmainc}}]
Assume that $k<28\cdot 2^{n-6}$; see \eqref{eqwhatisk}. 

Note at this point that \emph{no equality} will be assumed for $k$ before  \eqref{eqssmqvltwSx}. Therefore the numbered equations, equalities, and statements \emph{before}  \eqref{eqssmqvltwSx} can be used later in the proof of \ref{thmmain}\eqref{thmmaind}.  

We introduce the following notation. 
For a \ubta{} $\set{a,b}$, let
\begin{equation}
U(a,b):=\set{X\in \pset{\plu S}: a\in X,\text{ }b\in X,\text{ but }a\vee b\notin X };
\label{eqUabNotation}
\end{equation}
it is subset of $\pset{\plu S}$; note that 
the existence of $a\vee b$ above follows from \eqref{eqpbxCpBndxShR}. 
By Theorem~\ref{thmmain}\eqref{thmmaina},  $\sla$ is not a tree, whereby it has at least one \ubta. If it had only one \ubta, then Lemma~\ref{lemmabtggQPx} and Theorem~\ref{thmmain}\eqref{thmmainb} would imply that $k=28\cdot 2^{n-6}$. Hence, $\sla$ has at least two \ubta{}s. Let $\set{a_1,b_1}$, $\set{a_2,b_2}$, \dots, $\set{a_t,b_t}$ be a repetition-free list of all  \ubta{}s of $\sla$; note that $t\geq 2$. Let $v_i:=a_i\vee b_i$ for $i=1,\dots, t$.

First, we show that for any $1\leq i<j\leq t$,
\begin{align}
\text{if } |\set{a_i,b_i,v_i,a_j,b_j,v_j}|&=6,\text{ then } 
k\leq 24.5 \cdot 2^{n-6},\label{eqtxtdzRnYmQl6}\\
\text{if }|\set{a_i,b_i,v_i,a_j,b_j,v_j}|&= 5,\text{ then }
k\leq 25\cdot 2^{n-6},\text{ and}\label{eqtxtdzRnYmQl5}\\
\text{if }|\set{a_i,b_i,v_i,a_j,b_j,v_j}|&= 4,\text{ then }
k\leq 26\cdot 2^{n-6}.\label{eqtxtdzRnYmQl4}
\end{align}
In order to show this, let $U_i:=U(a_i,b_i)$; see \eqref{eqUabNotation}. That is, $U_i$ is  the set of all those $X\in \pset{\plu S}$ that contain $a_i$ and $b_i$ but not $v_i$. Then $U_i\cup U_j$ is disjoint from $\subplu$, whereby the Inclusion-Exclusion Principle, $k=|\subplu|$, $|\pset{\plu S}|=32\cdot 2^{n-6}$,  and $|U_i|=|U_j|=4\cdot 2^{n-6}$ give that
\begin{align}
\subplu&\subseteq \pset{\plu S}\setminus (U_i\cup U_j)
\text{, and so } \label{alignzcvngRta}\\
k  &\leq 2^{n-6}\cdot
(32-4-4)+|U_i\cap U_j|=24\cdot 2^{n-6} + |U_i\cap U_j|, \label{alignzcvngRtb}\\
&\text{and if \eqref{alignzcvngRta} holds with equality in it, then so does \eqref{alignzcvngRtb}.}\label{alignzcvngRtc}
\end{align}
The  equality in \eqref{eqtxtdzRnYmQl6}
 implies that $|U_i\cap U_j|\leq 2^{n-1-6}= 2^{n-7}$. Hence,  
\eqref{eqtxtdzRnYmQl6} follows from \eqref{alignzcvngRtb}. Similarly, \eqref{eqtxtdzRnYmQl5} follows from  \eqref{alignzcvngRtb} and from the fact that the equality in  \eqref{eqtxtdzRnYmQl5} gives that $|U_i\cap U_j|\leq 2^{n-1-5}= 2^{n-6}$. (Note that  $U_i\cap U_j$ maybe empty; for example, if $v_i=a_j$, then $|U_i\cap U_j|=0$.) If we assume the equality in 
\eqref{eqtxtdzRnYmQl4}, then $|U_i\cap U_j|\leq 2^{n-1-4}= 2\cdot 2^{n-6}$ and  \eqref{alignzcvngRtb} imply the validity of 
\eqref{eqtxtdzRnYmQl4} similarly.
Furthermore, it is clear from this argument that strict inequalities lead to strict inequalities. For later reference, we formulate this as follows.
\begin{equation}
\parbox{7.7cm}
{If $|U_i\cap U_j|$ is \emph{strictly} less than $2^{n-7}$, $2^{n-6}$, and $2\cdot 2^{n-6}$, then $k$ is strictly less than 
$24.5\cdot 2^{n-6}$,  $25\cdot 2^{n-6}$, and $26\cdot  2^{n-6}$, respectively.
}
\label{eqpbxdzhTGbTswS}
\end{equation}

Next, we claim that for $1\leq i<j\leq t$,
\begin{equation}
\text{if $v_i\neq v_j$, then $|\set{a_i,b_i,v_i,a_j,b_j,v_j}|\geq 5$.}
\label{eqczTrWvMd}
\end{equation}
In order to show this, first we deal with the case where  $v_j\in\set{a_i,b_i}$ or $v_i\in\set{a_j,b_j}$. Let, say, $v_1=a_2$. Then $v_2>a_2=v_1>a_1$ and  $v_2>a_2=v_1>b_1$ yield that  
$|\set{a_1,b_1,v_1,v_2}|=4$. Clearly, $b_2\notin\set{a_2=v_1, v_2}$. If we had that $b_2\in \set{a_1,b_1}$, then $v_2=a_2\vee b_2$ would belong to $\ideal v_1$, contradicting $v_1<v_2$. Hence, 
the inequality in \eqref{eqczTrWvMd} holds in this case.
Second, assume that $v_j\notin\set{a_i,b_i}$ and $v_i\notin\set{a_j,b_j}$. Using also that $v_i\neq v_j$, we have that $|\set{a_i,b_i,v_i,v_j}|=4$. Since $v_i\notin\set{a_j,b_j}$, $\set{a_i,b_i}\neq \set{a_j,b_j}$, and, of course, $v_j\notin\set{a_j,b_j}$, at least one of $a_j$ and $b_j$ is not in $\set{a_i,b_i,v_i,v_j}$, and the required inequality in \eqref{eqczTrWvMd} holds again. This proves \eqref{eqczTrWvMd}. Clearly, 
\begin{equation}
\text{if $v_i = v_j$ but $i\neq j$, then $|\set{a_i,b_i,v_i,a_j,b_j,v_j}|\geq 4$,}
\label{eqczTrsgPqD}
\end{equation}
because $\set{a_i,b_i,a_j,b_j}$ has at least three elements and does not contain $v_i=v_j$, which is strictly larger than every element of $\set{a_i,b_i,a_j,b_j}$.
Observe that $k\leq 26\cdot 2^{n-6}$, the first half of \ref{thmmain}\eqref{thmmainc}, follows from 
\eqref{eqtxtdzRnYmQl6},  \eqref{eqtxtdzRnYmQl5},  \eqref{eqtxtdzRnYmQl4}, \eqref{eqczTrWvMd}, and \eqref{eqczTrsgPqD},
 because $t\ge 2$ implies the existence of a pair $\pair i j$ such that $1\leq i<j\leq t$.

Next, strengthening \eqref{eqtxtdzRnYmQl5}, we are going to show that for any $1\leq i<j\leq t$,
\begin{equation}
\text{if } |\set{a_i,b_i,v_i,a_j,b_j,v_j}|=5 \text{ and }t\geq 3,\text{ then } 
k < 25 \cdot 2^{n-6}.\label{eqtxtstssPrzgdRnY}
\end{equation}
Since $t\geq 3$, we can pick an $m\in\set{1,\dots,t}\setminus\set{i,j}$. For the sake of contradiction,
\begin{equation}
\text{suppose that $|\set{a_i,b_i,v_i,a_j,b_j,v_j}|=5$ but $k\geq 25 \cdot 2^{n-6}$.}
\label{eqtxtIndRssmptn} 
\end{equation}
It follows from \eqref{eqpbxdzhTGbTswS} and \eqref{eqtxtIndRssmptn} that 
\begin{equation}
|U_i\cap U_j|=2^{n-6}.
\label{eqnbThslTVvgp}
\end{equation}
By \eqref{eqtxtdzRnYmQl6} and \eqref{eqtxtIndRssmptn}, none of $\set{a_i,b_i,v_i,a_m,b_m,v_m}$ and $\set{a_j,b_j,v_,a_m,b_m,v_m}$
consists of six elements. Using   \eqref{eqczTrWvMd} and \eqref{eqczTrsgPqD}, each of these sets consists of four or five elements. Hence,
\begin{equation}
|U_i\cap U_m|\leq 2\cdot 2^{n-6}\,\,\text{ and }\,\, |U_j\cap U_m|\leq 2\cdot 2^{n-6}.
\label{eqnbsklTVvqfjJ}
\end{equation}
We also need the following observation.
\begin{equation}
\parbox{7.5cm}{If $U_i\cap U_j\neq\emptyset$, $U_i\cap U_m\neq\emptyset$, and $U_j\cap U_m\neq\emptyset$, then $U_i\cap U_j\cap U_m\neq\emptyset$.}
\label{eqpbxiFthrmTlnGbwZTs}
\end{equation}
To show \eqref{eqpbxiFthrmTlnGbwZTs}, assume that its premise holds. If $\set{a_i,b_i,a_j,b_j,a_m,b_m}$ is disjoint from $\set{v_i,v_j,v_m}$, then $U_i\cap U_j \cap U_m$ contains $\set{a_i,b_i,a_j,b_j,a_m,b_m}$ and so it is nonempty. Otherwise, since $a$--$b$ symmetry and since the subscripts in \eqref{eqpbxiFthrmTlnGbwZTs} play symmetric roles, we can assume that 
$a_i=v_j$. However, then $U_i\cap U_j=\emptyset$, contradicting the premise of \eqref{eqpbxiFthrmTlnGbwZTs}. Consequently, \eqref{eqpbxiFthrmTlnGbwZTs} holds. Based on the Inclusion-Exclusion Principle, as in  \eqref{alignzcvngRta}--\eqref{alignzcvngRtc}, and using
\eqref{eqnbThslTVvgp} and \eqref{eqnbsklTVvqfjJ}, 
we can compute as follows; the overline and the underlines below will serve as reference points.
\begin{align}
&\begin{aligned}
k\leq 2^{n-6}\cdot\bigl(32-(4+4+4)\bigr)&+(|U_i\cap U_j|+|U_i\cap U_m|+|U_j\cap U_j|) \cr
&- |U_i\cap U_j\cap U_m|,\quad\text{ and so}
\end{aligned}\label{eqsztfRml}\\
&\begin{aligned}
k&\leq 2^{n-6}\cdot(20+\overline 1+\underbar 2+\underbar 2) - |U_i\cap U_j\cap U_m|\cr
&=25\cdot 2^{n-6} - |U_i\cap U_j\cap U_m|.
\end{aligned}
\label{eqalignZtBjWsCpPSPw}
\end{align}
We know from \eqref{eqnbThslTVvgp} that $U_i\cap U_j\neq\emptyset$. The underlined numbers in \eqref{eqalignZtBjWsCpPSPw} come from \eqref{eqnbsklTVvqfjJ}. So  if at least one the intersections  $U_i\cap U_m$ and $U_j\cap U_m$ is empty, then at least one of the underlined numbers can be replaced 0 and  \eqref{eqalignZtBjWsCpPSPw} gives that $k<25\cdot 2^{n-6}$. Otherwise the subtrahend at the end of \eqref{eqalignZtBjWsCpPSPw} is positive by  \eqref{eqpbxiFthrmTlnGbwZTs}, and we obtain again that $k<25\cdot 2^{n-6}$.

This contradicts \eqref{eqtxtIndRssmptn} and proves the validity of \eqref{eqtxtstssPrzgdRnY}. Next, we assume that
\begin{equation}
k = 26\cdot 2^{n-6}.
\label{eqssmqvltwSx}
\end{equation}
It follows from  \eqref{eqtxtdzRnYmQl6},  \eqref{eqtxtdzRnYmQl5},  \eqref{eqczTrWvMd}, and \eqref{eqssmqvltwSx} that 
\begin{equation}
\text{all the $v_i$ are the same, so we can let $v:=v_1=\dots=v_t$.}
\label{eqtxtallviareThesame}
\end{equation} 
Hence, we conclude from  \eqref{eqtxtdzRnYmQl6},  \eqref{eqtxtdzRnYmQl5}, and  \eqref{eqczTrsgPqD} that, for any  $1\leq i<j\leq t$,
\begin{equation}
\parbox{7.8cm}{$|\set{a_i,b_i,a_j,b_j,v}|= 4$\text{ and so }$|U_i\cap U_j|\leq 2\cdot 2^{n-6}$ and 
$|\set{a_i,b_i}\cap\set{a_j,b_j}|=1$.}
\label{eqczsdfjsgTzNP}
\end{equation}

Next, we are going to prove that 
$t$, the number of \ubta{}s, equals $2$.
Suppose the contrary. Since now we have \eqref{eqczsdfjsgTzNP} instead of \eqref{eqnbThslTVvgp}, $\overline 1$ and $25$ in \eqref{eqalignZtBjWsCpPSPw} turns into $\underbar 2$ and $26$, 
respectively. These two modifications do not influence the paragraph following \eqref{eqalignZtBjWsCpPSPw}, and we conclude that the inequality in the modified \eqref{eqalignZtBjWsCpPSPw} is a strict one, that is, $k<26\cdot 2^{n-6}$. This contradicts \eqref{eqssmqvltwSx}, whence we conclude 
that there are exactly $t=2\,$ \ubta{}s. We know from \eqref{eqczsdfjsgTzNP} that they are not disjoint. So we can denote them by $\set{a,b}$ and $\set{c,b}$ where $|\set{a,b,c}|=3$. 
By \eqref{eqtxtallviareThesame}, 
$v=a\vee b=c\vee b$. We know from $t=2$ that $\set{a,c}$ is not a \ubta{}, whence $a$ and $c$ are comparable. So we can assume that $a<c$, and it follows from Lemma~\ref{lemmabtwWsPtZ} that $\sla$ is a quasi-tree semilattice of the required form.

Finally, assume that $\sla$ is a quasi-tree semilattice
and its nucleus is the pentagon $N_5=\set{u,a,b,c,v}$ with bottom $u$, top $v$, and $a<c$.
Let $U_1:=U(a,b)$ and $U_2:=U(c,b)$; see \eqref{eqUabNotation}.
 Since $\subplu=\pset{\plu S}\setminus(U_1\cup U_2)$,
\begin{equation*}
k=|\pset{\plu S}|-|U_1|-|U_2|+|U_1\cap U_2|=(32-4-4+2)\cdot 2^{n-6}=26\cdot 2^{n-6},
\end{equation*}
as required. This completes the proof of Theorem~\ref{thmmain}\eqref{thmmainc}.
\end{proof}

\begin{lemma}\label{lemmabowtie}
If $\sla$ from \eqref{eqwhatisk} contains  exactly two \ubta{}s, 
$\set{a,b}$ and $\set{b,c}$ such that $v_1:=a\vee b$ and $v_2:=b\vee c$ are incomparable, then $\sla$ is a quasi-tree semilattice and its nucleus is $F=\set{u:=a\wedge b\wedge c, a,b,c, v_1,v_2}$ given in Figure~\textup{\ref{figfour}}. 
\end{lemma}

\begin{proof}[Proof of Lemma~\textup{\ref{lemmabowtie}}] Let $u:=a\wedge b$. It is not in $\set{a,b}$. Since $b\ngeq c$, we have that $u\ngeq c$. Using that $v_2$ is an upper bound of $\set{u,c}$ and
$\set{u,c}$ is not a \ubta, it follows that $\set{u,c}$ is not an antichain. Hence, $u\leq c$, whence $u=a\wedge b\wedge c$. 
The set $\set{b,a\wedge c}$ cannot be an antichain, since otherwise it would be an additional \ubta{} with upper bound $v_1$. Since $b\nleq c$, we have that $b\nleq a\wedge c$. Hence, $
a\wedge c=a\wedge c\wedge b$. Summarizing the facts above and taking into account that $a$ and $c$ play a symmetric role, we have that 
\begin{equation}
u=a\wedge b=a\wedge b\wedge c=b\wedge c=a\wedge c.
\label{eqmTsZrTF}
\end{equation}
Let $M:=\set{a,b,c,u,v_1,v_2}$; we claim that 
\begin{equation}
\text{$M$ is a convex meet-subsemilattice of $\sla$.}
\label{eqtxtMisconvex}
\end{equation}
It is a meet subsemilattice by \eqref{eqmTsZrTF}.
For the sake of contradiction,suppose that $x\in S\setminus M$ such that $u< x < v_1$; the case $u< x < v_2$ would be similar since $a$ and $c$ play symmetric roles. 
Both $\set{a,x}$ and $\set{x,b}$ are have an upper bound, $v_1$. Hence, none of them is an \ubta{} since $x\notin M$. Hence, 
$a\leq x \leq b$, or $b\leq x\leq a$, or $a,b\in\ideal x$, or $a,b\in\filter x$. The first two alternatives are ruled out by $a\parallel b$. The third alternative leads to $v_1=a\vee b \leq x\leq v_1$, contradicting $x\notin M$. We obtain a contradiction from the fourth alternative dually by using $u$ instead of $v_1$. Thus, \eqref{eqtxtMisconvex} holds. It is clear by \eqref{eqmTsZrTF} that $M\cong F$. 

Since $\pair u {v_1}=\pair{a\wedge b}{a\vee b}$ occurs in \eqref{eqtrCnGr} and the $\tcon$-blocks are convex subsets, 
$\set{a,b,v_1,u}\subseteq \blokk u\tcon$. We obtain similarly that $\set{b,c, v_2,u}\subseteq \blokk u\tcon$, whence we have that $M\times M\subseteq \tcon$. Therefore, since $M$ contains both \ubta{}s and their joins, 
 Lemma~\ref{lemmauGbtThnNZcls} implies the validity of Lemma~\ref{lemmabowtie}.
\end{proof}

\begin{lemma}\label{lemmaNsix}
If $\sla$ from \eqref{eqwhatisk} contains  exactly three \ubta{}s, 
$\set{a_1,b}$, $\set{a_2,b}$, and $\set{a_3,b}$ such that $v:=a_1\vee b=a_2\vee b=a_3\vee b$ and $a_1<a_2<a_3$, then $\sla$ is a quasi-tree semilattice and its nucleus is $N_6=\set{u:=a_1\wedge b=a_2\wedge b=a_3\wedge b, a_1,a_2,a_3,v}$ given in Figure~\textup{\ref{figfour}}.
\end{lemma}

\begin{proof}[Proof of Lemma~\textup{\ref{lemmaNsix}}]
Let $u:=a_3\wedge b$; clearly, $u\neq b$.  We are going to show that $M:=\set{u, a_1,a_2,a_3,v}$ is a subsemilattice isomorphic to $N_6$. 
Let $i\in\set{1,2}$.   Since $v$ is an upper bound of the set $\set{a_i, u}$, this set is not an antichain. Since $a_i\nleq b$, we have that   $a_i\nleq u$. Hence, $u<a_i$, and we obtain that 
$u\leq a_i\wedge b\leq a_3\wedge b=u$.  Thus, the meets in
$M$ are what they are required to be, and we conclude that $M\cong N_6$. Next, for the sake of contradiction, suppose that  $M$ is not a convex subset of $\sla$, and pick an element $x\in S\setminus M$ such that $u\leq x\leq v$. Since no more \ubta{} is possible, none of $a_1$, $a_2$, $a_3$, and $b$ is incomparable with $x$. If we had that $x\leq a_j$ for some $j\in\set{1,2,3}$, then $b\leq x$ would contradict
$a_j\nleq b$ while $x\leq b$ would lead to $u\leq x\wedge b\leq a_j\b=u$, a contradiction since $x\neq u\in M$. A dual argument, with $v$ instead of $u$, would lead to a contradiction if $a_j\leq x$. Hence, $M$ is a convex subsemilattice of $\sla$. 
Since $\pair u v=\pair{a_1\wedge b}{a_1\vee b}$ occurs in \eqref{eqtrCnGr} and the $\tcon$-blocks are convex subsets, $M\times M\subseteq \tcon$. Therefore, since $M$ contains all the three \ubta{}s and their common join,  Lemma~\textup{\ref{lemmaNsix}} follows from Lemma~\ref{lemmauGbtThnNZcls}.
\end{proof}

\begin{proof}[Proof of Theorem~\textup{\ref{thmmain}\eqref{thmmaind}}] We assume that $k=|\Con(S;\wedge)|<26\cdot 2^{n-6}$.
In the first part of the proof, we are going to focus on the required inequality, $k\leq 25\cdot 2^{n-6}$.

As it has been mentioned in the previous proof, any part of that proof before \eqref{eqssmqvltwSx} is applicable here, including the notation.
If  $|\set{a_i,b_i,v_i,a_j,b_j,v_j}|\geq 5$ 
or $v_i\neq v_j$ for some $1\leq i<j\leq t$, then 
the required $k\leq 25\cdot 2^{n-6}$
follows from \eqref{eqtxtdzRnYmQl6}, \eqref{eqtxtdzRnYmQl5}, and  \eqref{eqczTrWvMd}. Otherwise, we can assume that 
that $v:=v_1=v_2=\cdots=v_t$, and combining   \eqref{eqtxtdzRnYmQl5} and \eqref{eqczTrsgPqD}, we can also assume that 
$|\set{a_i,b_i,a_j,b_j,v}|=4$ for or  all $1\leq i<j\leq t$.
For later reference, we summarize this assumption as
\begin{equation}
\parbox{8.5cm}{$v:=v_1=v_2=\cdots=v_t$ and $|\set{a_i,b_i,a_j,b_j,v}|=4$,\\ whereby
$|\set{a_i,b_i}\cap\set{a_j,b_j}|=1$,   for  all $1\leq i<j\leq t$.  
}
\label{eqdzRpQbMnsS}
\end{equation}
We claim that
\begin{equation}
\parbox{8.5cm}{if $t\geq 3$, \eqref{eqdzRpQbMnsS}, and $\set{a_1,b_1}\cap\set{a_2,b_2}\cap\set{a_3,b_3}=\emptyset$, then $k\leq 24\cdot 2^{n-6}$.}
\label{eqtxtczhBmNQwx}
\end{equation}
The pairwise intersections in \eqref{eqdzRpQbMnsS}  are singletons, whereby the only way that the intersection in \eqref{eqtxtczhBmNQwx} is empty is that $|\set{a_1,b_1,a_2,b_2,a_3,b_3}|=3$. Hence,
$|U_i\cap U_j|=|U_1\cap U_2\cap U_3|=2\cdot 2^{n-6}$, and 
\eqref{eqtxtczhBmNQwx} follows from \eqref{eqsztfRml}. We also claim that
\begin{equation}
\parbox{8.5cm}{if $t\geq 3$, \eqref{eqdzRpQbMnsS}, and $\set{a_1,b_1}\cap\set{a_2,b_2}\cap\set{a_3,b_3}\neq\emptyset$, then $k\leq 25\cdot 2^{n-6}$.}
\label{eqtxtczkhBserkTlW}
\end{equation}
With the assumption made in \eqref{eqtxtczkhBserkTlW}, if we consider the same intersections as in the argument right after \eqref{eqtxtczhBmNQwx}, then we obtain that $|\set{a_1,b_1,a_2,b_2,a_3,b_3}|=4$. Hence, 
$|U_i\cap U_j|=2\cdot 2^{n-6}$ and 
$|U_1\cap U_2\cap U_3|=1\cdot 2^{n-6}$, and \eqref{eqtxtczkhBserkTlW} follows from \eqref{eqsztfRml}. Our next observation is that
\begin{equation}
\text{if $t\leq 2$ and \eqref{eqdzRpQbMnsS},  then $k\geq 26\cdot 2^{n-6}$.}
\label{eqtxsjtGynfNhZgsZtTs}
\end{equation}
For $t\leq 1$, this is clear from Theorem~\ref{thmmain}\eqref{thmmaina}, Lemma ~\ref{lemmabtggQPx}, and Theorem~\ref{thmmain}\eqref{thmmainb}; so let $t=2$.  Since the intersection in \eqref{eqdzRpQbMnsS} is a singleton, the two \ubta{}s are of the form $\set{a,b}$ and $\set{c,b}$. Since $\set{a,c}$ cannot be a third \ubta{}, the elements $a$ and $c$ are comparable, whereby Lemma~\ref{lemmabtwWsPtZ}, and Theorem~\ref{thmmain}\eqref{thmmainc} imply that $k = 26\cdot 2^{n-6}$. Thus, \eqref{eqtxsjtGynfNhZgsZtTs} holds.
Now, the required $k\leq 25\cdot 2^{n-6}$ follows from \eqref{eqtxtczhBmNQwx}, \eqref{eqtxtczkhBserkTlW},  \eqref{eqtxsjtGynfNhZgsZtTs}, and the paragraph above \eqref{eqdzRpQbMnsS}; completing the first part of the proof.

In the rest of the proof, we will always assume that $k = 25
\cdot 2^{n-6}$, even if this is not emphasized all the time.
We claim that 
\begin{equation}
\parbox{9cm}{if $k = 25\cdot 2^{n-6}$ and $t\geq 3$, then 
$t=3$, $v:=v_1=\dots=v_t,$ and \eqref{eqczsdfjsgTzNP} holds for all $1\leq i<j\leq t$.}
\label{eqpbxhpHzHhZbGswR} 
\end{equation}
We obtain from \eqref{eqtxtdzRnYmQl6} that the size of $\set{a_i,b_i,u_i,a_j,b_j,v_j}$ is not 6. We obtain from \eqref{eqtxtstssPrzgdRnY} that it is neither 5, whereby this size is 4 since $\set{a_i,b_i}\neq\set{a_j,b_j}$. Thus,
 \eqref{eqczTrWvMd} implies  the validity of \eqref{eqczsdfjsgTzNP} and $v_1=\cdots=v_n$, which we denote by $v$. The $|\set{a_i,b_i}\cap\set{a_j,b_j}|=1$ part of \eqref{eqczsdfjsgTzNP} implies that, apart from notation (to be more exact, apart from $a$--$b$ symmetry),
\begin{equation}
\parbox{9.4cm}{whenever $1\leq i<j<m\leq m$, then either
$b_i=a_j$, $b_j=a_m$, and $b_m=a_i$, or $b:=b_i=b_j=b_m$ and $|\set{a_i,a_j,a_m}|=3$.}
\label{eqpbxhThnBxVpwhsn}
\end{equation}
It follows similarly to \eqref{eqsztfRml} and \eqref{eqalignZtBjWsCpPSPw} that
\begin{equation}
\parbox{9.0cm}{if the first alternative of \eqref{eqpbxhThnBxVpwhsn}  holds, 
then $|U_i\cup U_j\cup U_m|=\bigl((4+4+4)-(2+2+2)+2\bigr)\cdot 2^{n-6}$, whereby $k\leq (32-8)\cdot 2^{n-6}$, which contradicts $k = 25 \cdot 2^{n-6}$,}
\label{eqpbxdzTnBpWslwD}
\end{equation}
since $U_i\cap U_j\cap U_m=U_i\cap U_j$. Thus, \eqref{eqpbxdzTnBpWslwD} excludes the first alternative of \eqref{eqpbxhThnBxVpwhsn}. 
Hence we have the second alternative,  $|U_i\cap U_j\cap U_m|= 2^{n-6}$, and it follows similarly to \eqref{eqsztfRml} and \eqref{eqalignZtBjWsCpPSPw} that
\begin{equation}
\text{$|U_i\cup U_j\cup U_m|=\bigl((4+4+4)-(2+2+2)+1\bigr)\cdot 2^{n-6}=7 \cdot 2^{n-6}$.}
\label{eqtxtxdzTsPrnnmDslwskK}
\end{equation}
Now, for the sake of contradiction, suppose that $t\geq 4$. Then we can and pick an index  $s\in\set{1,\dots,t}\setminus\set{i,j,m}$. The  \ubta{} $\set{a_s,b_s}$ belongs to $U_s$ but it does not belong $U_i$ since the members of $U_i$ contain both $a_i$ and $b_i$ but $\set{a_s,b_s}\neq \set{a_i,b_i}$. Similarly, 
$\set{a_s,b_s}$ belongs neither to $U_j$, nor to $U_m$, whence it is not in $U_i\cup U_j\cup U_m$. Hence, $U_i\cup U_j\cup U_m$ is a proper subset of $U_i\cup U_j\cup U_m\cup U_s$, which is disjoint from $\subplu$. Thus, by \eqref{eqtxtxdzTsPrnnmDslwskK}, strictly more than $7 \cdot 2^{n-6}$ subsets of $\plu S$ are  \emph{not} in $\subplu$, and we obtain that 
$k=|\subplu|< (32-7)\cdot 2^{n-6}$. This contradicts $k=25\cdot 2^{n-6}$ and excludes that $t\geq 4$. Thus, $t=3$ and we have proved \eqref{eqpbxhpHzHhZbGswR}.

Next, assume that $t\geq 3$. We know from \eqref{eqpbxhpHzHhZbGswR} that $t=3$. Furthermore, by
\eqref{eqpbxhpHzHhZbGswR}, \eqref{eqpbxhThnBxVpwhsn}, and \eqref{eqpbxdzTnBpWslwD}, $\set{a_1,b}$, 
$\set{a_2,b}$, and $\set{a_3,b}$ is the list of all \ubta{}s with a common join $v$. No two of $a_1$, $a_2$, and $a_3$ are  incomparable, since otherwise those two would form a \ubta{} (with upper bound $v$). Hence, we can assume that $a_1<a_2<a_3$. Thus, it follows from Lemma~\ref{lemmaNsix} that $\sla$ is a quasi-tree semilattice with nucleus $N_6$.

Finally, assume that $t\ngeq 3$. By
Theorem~\ref{thmmain}\eqref{thmmaina}--\eqref{thmmainb} and  Lemma~\ref{lemmabtggQPx}, $t\notin\set{0,1}$, whence $t=2$. 
There are several cases to consider.
 
\begin{case}[we assume that $v_1=v_2$ and $\set{a_1,b_1}\cap\set{a_2,b_2}\neq \emptyset$]\label{caseone} By $a$--$b$ symmetry, we can choose the notation so that $a:=a_1$, $b:=b_1=b_2$, and $c:=a_2$. If $a\parallel c$, then $\set{a,c}$ is a third \ubta{} (with upper bound $v_1=v_2$), contradicting $t=2$. Hence, we can assume that $a < c$. But then,  by 
 Lemma~\ref{lemmabtwWsPtZ}, $\sla$ is a quasi-tree semilattice  with nucleus $N_5$, and so  \ref{thmmain}\eqref{thmmainc} gives that $k=26\cdot 2^{n-6}$, a contradiction again since $k=25\cdot 2^{n-6}$ has been assumed. So Case~\ref{caseone} cannot occur.
\end{case} 

\begin{case}[we assume that $v_1=v_2$ and $\set{a_1,b_1}\cap\set{a_2,b_2} = \emptyset$]\label{casetwo} Observe that for 
every $X\subseteq\set{a_1,b_1,a_2,b_2}$ such that  $|X|=2$,
\begin{equation}
\text{if $\set{a_1,b_1}\neq X\neq \set{a_2,b_2}$, then $X$ is not an antichain,}
\label{eqtxtHnbTpxCvWwW}
\end{equation}
since otherwise $X$ would be a third \ubta{} with upper bound $v_1=v_2$. By 1--2 symmetry, we can assume that $a_1<a_2$. 
By \eqref{eqtxtHnbTpxCvWwW}, $a_2$ and $b_1$ are comparable elements. If we had that $a_2\leq b_1$, then we would obtain $a_1\leq b_1$ by transitivity, contradicting that $\set{a_1,b_1}$ is a \ubta{}.
Hence,   $b_1< a_2$. But then the inequality in $v_1=a_1\vee b_1\leq a_2<v_2=v_1$ is a contradiction. Therefore, Case~\ref{casetwo} cannot occur.
\end{case}

Cases~\ref{caseone} and \ref{casetwo} make it clear that now, when $t=2$, we have that $v_1\neq v_2$. We obtain from \eqref{eqtxtdzRnYmQl6} and \eqref{eqczTrWvMd} that 
\begin{equation}
|\set{a_1,b_1,v_1,a_2,b_2,v_2}| = 5.
\label{eqnwxltlfV}
\end{equation}
The following two cases have to be dealt with.

\begin{case}[we assume that $v_1\neq v_2$ and $\set{a_1,b_1,a_2,b_2}\cap\set{v_1,v_2}=\emptyset$]\label{casethree}
This assumption and  \eqref{eqnwxltlfV} allow us to assume that $\set{a_1,b_1}=\set{a,b}$ and $\set{a_2,b_2}=\set{c,b}$. So $v_1=a\vee b$ and $v_2=c\vee b$. For the sake of contradiction, suppose that  $a$ and $c$ are comparable. Let, say, $a<c$; then $v_1=a\vee b \leq c\vee b= v_2$. But $v_1\neq v_2$, so $v_1<v_2$. If we had that $c\leq v_1$, then $v_2= b\vee c\leq v_1$ would 
contradict $v_1<v_2$. 
If we had that $v_1\leq c$, then this would lead to the contradiction $b\leq c$ by transitivity. Hence, $c\parallel v_1$. So $\set{c,v_1}$ is an additional \ubta{} (with upper bound $v_2$), which is a contradiction showing that $a\parallel c$.
If $v_1$ and $v_2$ were comparable, then the larger one of them would be an upper bound of $\set{a,c}$, and so $\set{a,c}$ would be a third \ubta. 
Thus, $v_1\parallel v_2$, and Lemma~\ref{lemmabowtie} gives that $\sla$ is a quasi-tree semilattice with nucleus $F$, as required.
\end{case}

\begin{case}[we assume that $v_1\neq v_2$ and $\set{a_1,b_1,a_2,b_2}\cap\set{v_1,v_2}\neq\emptyset$]\label{casefour}
Since $a$ and $b$ play symmetric roles and so do the subscripts 1 and 2, we can assume that $v_1=a_2$. We have that $|\set{a_1,b_1,a_2,b_2}|=4$ since $b_2\nleq a_2=v_1$ excludes the possibility that $b_2\in\set{a_1,b_1,a_2}$.  
None of the sets $\set{a_1,b_2}$ and $\set{b_1,b_2}$ is an antichain, since otherwise the set in question would be a new \ubta{} with upper bound $v_2$, which would be contradiction.
Hence, $a_1$ and $b_2$ are comparable elements, and so do $b_1$ and $b_2$. If we had that $a_1\geq b_2$ or $b_1\geq b_2$, then 
transitivity would lead to $a_2=v_1\geq b_2$, a contradiction. Thus, $a_1\leq b_2$ and $b_1\leq b_2$. But then $a_2=v_1=a_1\vee b_1\leq b_2$ is  a contradiction. This shows that Case~\ref{casefour} cannot occur.
\end{case}

Now that all cases have been considered,  we have shown that 
 if $k=25\cdot 2^{n-6}$, then $\sla$ is of the required form.

Finally, if $\sla$ is a quasi-tree semilattice with nucleus $N_6$, then using the Inclusion-Exclusion Principle as in \eqref{eqsztfRml} and  \eqref{eqalignZtBjWsCpPSPw}, we obtain that 
\begin{align*}
|\Con(S;\wedge)|=2^{n-6}\bigl(20+(2+2+2)-1\bigr)=25\cdot 2^{n-6},
\end{align*}
as required. Similarly, if the nucleus is $F$, then we follow the method of \eqref{alignzcvngRta} and \eqref{alignzcvngRtb} to obtain the required 
\begin{align*}
|\Con(S;\wedge)|=2^{n-6}\bigl(32-(4+4)+1\bigr)=25\cdot 2^{n-6}.
\end{align*}
This completes the proof of Theorem~\ref{thmmain}\eqref{thmmaind}.
\end{proof}

\end{document}